\documentclass[ 11pt]{article}

\usepackage[utf8]{inputenc}  
\usepackage[T1]{fontenc}         
\usepackage{geometry}         
\usepackage[english]{babel}  
\usepackage{cite}
\usepackage{amssymb}   
\usepackage{amsmath} 
\usepackage{amscd}
\usepackage{graphicx, color}    
\usepackage{amsthm}                 
\usepackage{latexsym}     
\usepackage[all]{xy}

\newtheorem{thm}{Theorem}[section]
\newtheorem{cor}[thm]{Corollary}
\newtheorem{prop}[thm]{Proposition}
\newtheorem{remark}[thm]{Remark}
\newtheorem{lem}[thm]{Lemma}

\theoremstyle{definition}

\newcommand{\bbZ}{{\mathbb Z}}

\newcommand{\ra}{\rightarrow}

\title{\textbf{Hyperbolic groups with homeomorphic Gromov boundaries.}}           
\author{Alexandre Martin \& Jacek Świątkowski\footnote{The second author is partially supported by the Polish Ministry of Science and Higher Education (MNiSW), Grant N201 541738.}}
\date{}

\begin{document}
\maketitle

\noindent {\bf Abstract.} We show that the Gromov boundary of the free product of two infinite hyperbolic groups is uniquely determined up to homeomorphism by the homeomorphism types of the boundaries of its factors. We generalize this result to graphs of hyperbolic groups over finite subgroups. Finally, we give a necessary and sufficient condition for the Gromov boundaries of any two hyperbolic groups to be homeomorphic (in terms of the topology of the boundaries of factors in terminal splittings over finite subgroups).

\section{Introduction}

It is well known that the Gromov boundary of the free product $G_1*G_2$ of two infinite hyperbolic groups is some combination of the Gromov boundaries of its factors. However, the nature of this combination seems to be not clarified. In particular, it is not clear wheather the topology of $\partial(G_1*G_2)$ is uniquely determined by the topology of $\partial G_1$ and $\partial G_2$.
In this paper we show, among other things, that the answer to the above question is positive.

In \cite{PapasogluWhyteQI}, P. Papasoglu and K. Whyte show that, up to quasi-isometry, such a free product is uniquely determined by the quasi-isometry types of its factors. Since there are hyperbolic groups whose boundaries are homeomorphic, but the groups themselves are not quasi-isometric, this result does not help to answer the question above.

Several descriptions of the boundary $\partial(G_1*G_2)$ appear more or less explicitely in the literature (see e.g. \cite{TirelProducts}, \cite{WoessMartinBoundary}). However, they depend a priori on the groups $G_i$, not only on the topology of their Gromov boundaries, and hence the question requires a more careful analysis.

\bigskip
We now present the main results of the paper, which generalize the above mentioned result concerning the free product. Recall that the fundamental group $\pi_1{\cal G}$ of a graph of groups $\cal G$ with all vertex groups hyperbolic and all edge groups finite is itself hyperbolic.

\begin{thm}
For $i=1,2$, let ${\cal G}_i$ be graphs of groups with all vertex groups hyperbolic and all edge groups finite, and suppose that the groups $\pi_1{\cal G}_i$ have infinitely many ends. Denote by $h({\cal G}_i)$ the set of homeomorphism types of Gromov boundaries of those vertex groups in ${\cal G}_i$ which are nonelementary hyperbolic (i.e. are not finite and not virtually cyclic). If $h({\cal G}_1)=h({\cal G}_2)$ then the Gromov boundaries $\partial(\pi_1({\cal G}_i))$ are homeomorphic.
\label{Thm1}
\end{thm}

Recall that a hyperbolic group is 1-ended if its Gromov boundary is a nonempty connected space. Theorem 1.1 has the following partial converse. 

\begin{thm}
Under assumptions and notation as in Theorem \ref{Thm1}, suppose additionally that all the vertex groups of both ${\cal G}_i$ are either finite or 1-ended. If the Gromov boundaries $\partial(\pi_1({\cal G}_i))$ are homeomorphic then the sets $h({\cal G}_i)$ of homeomorphism types of 1-ended vertex groups in ${\cal G}_i$ are equal. 
\label{Thm2}
\end{thm}

Recall that, by a result of M. Dunwoody \cite{DunwoodyAccessibility}, each finitely presented group $G$ has {\it a terminal splitting over finite subgroups}, i.e. it is isomorphic to the fundamental group $\pi_1({\cal G})$ of a graph of groups ${\cal G}$ whose vertex groups are 1-ended or finite and whose edge groups are finite. If $G$ is hyperbolic, all vertex groups in any such splitting are also hyperbolic.
The next corollary is a direct consequence of Theorems \ref{Thm1} and \ref{Thm2}.

\begin{cor}
For $i=1,2$, let ${\cal G}_i$ be terminal splittings over finite subgroups of hyperbolic groups $G_i$ with infinitely many ends. Then the Gromov boundaries $\partial G_i$ are homeomorphic iff we have $h({\cal G}_1)=h({\cal G}_2)$ (i.e. the sets of homeomorphism types of boundaries of 1-ended vertex groups in ${\cal G}_i$ coincide).
\end{cor}

\bigskip
The paper is organized as follows. In Sections 2-4 we deal with the case of the free product $G=G_1*G_2$ of two groups. More precisely, in Section 2 we 
describe some topological space $\delta\Gamma=\delta\Gamma(G_1,G_2)$, arranged out of copies of the Gromov boundaries $\partial G_1$ and $\partial G_2$, and equipped with a natural action of $G$. In Section 3 we show that $\delta\Gamma$ is $G$-equivariantly homeomorphic to the Gromov boundary $\partial G$. We do not claim priority for the content of those two sections, as similar ideas seem to be known to various researchers (see Remark 3.2 and comments at the beginning of Subsection 2.2).
In Section 4 we prove Theorem \ref{Thm1} in the case of free product of two groups (see Theorem \ref{homeofreeprod}). In Section 5 we deduce full Theorem \ref{Thm1} from the special case considered in Section 4, and from the result of Papasoglu and Whyte \cite{PapasogluWhyteQI}. Finally, in Section 6 we prove Theorem \ref{Thm2}.

\section{Geometric model for a free product.}

\subsection{Tree of spaces associated to a free product.}

\noindent \textbf{Bass-Serre tree of the splitting.}  
Let $G=A*B$ be the free product of two hyperbolic groups $A$ and $B$. Let $\tau$ be a copy of the unit interval, and denote by $v_A$ and $v_B$ its vertices. We define a tree $T=T(A,B)$, called the \textit{Bass-Serre tree} of the splitting, as $G \times \tau$ divided by the equivalence relation $\sim$ induced by the equivalences 
$$(g_1,v_A) \sim (g_2, v_A)\,\, \mbox{ if }\,\, g_2^{-1}g_1 \in A,$$
$$(g_1,v_B) \sim (g_2, v_B)\,\, \mbox{ if }\,\, g_2^{-1}g_1 \in B.$$
This tree comes with an action of $G$. 
More precisely, if $[g,x]$ denotes the equivalence class of $(g,x)\in G\times\tau$,
we set $g\cdot[g',x]=[gg',x]$.
A  strict fundamental domain for this action is any edge
$[g,\tau]=\{ [g,x]:x\in\tau\}$. Edges have trivial pointwise stabilizers, while vertex stabilizers are the conjugates in $G$ of the factors $A$ and $B$. For more information about Bass-Serre trees we refer the reader to \cite{SerreTrees}.\\

\noindent \textbf{Tree of Cayley graphs.} To get a convenient geometric model for $G$ we ``blow up'' the vertices of $T$, by replacing them (in an equivariant way) with copies of the Cayley graphs of the associated vertex stabilizers. Let $\Gamma A, \Gamma B$ be the Cayley graphs of $A$ and $B$, with respect to some chosen finite sets of generators, with distinguished vertices as basepoints. Define an auxiliary graph $X$ as the union of $\Gamma A$, $\Gamma B$ and $\tau$, where the vertex $v_A$ of $\tau$ is identified with the basepoint of $\Gamma A$ and $v_B$ with the basepoint of $\Gamma B$. Define a graph $\Gamma=\Gamma(A,B)$ as $G \times X$ divided by the equivalence relation induced by
$$
(g_1,x_1) \sim (g_2, x_2) \mbox{\quad if \quad} x_1, x_2 \in \Gamma A, \,\,\, g_2^{-1}g_1 \in A \,\,\, \mbox{ and }\,\,\, g_2^{-1}g_1x_1 = x_2,
$$
$$
(g_1,x_1) \sim (g_2, x_2) \mbox{\quad if \quad} x_1, x_2 \in \Gamma B,\,\,\, g_2^{-1}g_1 \in B \,\,\, \mbox{ and }\,\,\, g_2^{-1}g_1x_1 = x_2.
$$
The canonical map $X \ra \tau$ obtained by collapsing each Cayley graph on its basepoint extends to an equivariant continuous map $\Gamma\ra T$, which equips $\Gamma$ with a structure of tree of spaces (compare \cite{ScottWall}). We denote by $\Gamma G_v$ the preimage of a vertex $v$ of $T$ under this map. This notation agrees with the fact that $\Gamma G_v$ is a subgraph of $\Gamma$ isomorphic to the Cayley graph of the stabilizing subgroup $G_v$ of $G$ in its action on $T$
(which is an appropriate conjugate of either $A$ or $B$ in $G$). Thus, the structure of tree of spaces over $T$ for $\Gamma$ consists of the subgraphs $\Gamma G_v$ corresponding to the vertices of $T$, and of edges connecting them. These connecting edges are in a natural bijective correspondence with the edges of $T$, and we call them {\it lifts} of the corresponding edges of $T$ under the above projection map $\Gamma\to T$.


\subsection{Compactification and boundary of $\Gamma$.}
We now define a compactification $\overline{\Gamma}$ of $\Gamma$, using Gromov boundaries $\partial A$, $\partial B$ of the groups $A$ and $B$ as ingredients, and its boundary $\delta\Gamma=\overline\Gamma\setminus\Gamma$. This boundary consists of two disjoint sets, corresponding to the two ways of ``aproaching infinity'' in the tree of spaces $\Gamma$. Similar descriptions, in slightly different contexts and expressed in different terms (e.g. by an explicit metric or by description of convergent sequences), can be found in \cite{TirelProducts} and \cite{WoessMartinBoundary}. Our description is inspired by a construction given in \cite{DahmaniCombination}, which does not apply directly to our case. \\

\noindent \textbf{Boundaries of stabilizers.}  Let $\delta_{Stab}\Gamma$ be the set $G \times \big( \partial A \cup \partial B \big)$ divided by the equivalence relation induced by
$$
(g_1,\xi_1) \sim (g_2, \xi_2) \mbox{\quad if \quad} \xi_1, \xi_2 \in \partial A, \,\,\, g_2^{-1}g_1 \in A \,\,\, \mbox{ and }\,\,\, g_2^{-1}g_1\xi_1 = \xi_2,
$$
$$
(g_1,\xi_1) \sim (g_2, \xi_2) \mbox{\quad if \quad} \xi_1, \xi_2 \in \partial B, \,\,\, g_2^{-1}g_1 \in B \,\,\, \mbox{ and }\,\,\, g_2^{-1}g_1\xi_1 = \xi_2.
$$
We denote by $[g, \xi]$ the equivalence class of an element $(g, \xi)$. The set $\delta_{Stab}\Gamma$ comes with a natural action of $G$ given by $g'\cdot[g,\xi]=[g'g,\xi]$. It also comes with a natural projection to the set ${\cal V}(T)$ of vertices of the Bass-Serre tree $T$. We denote by $\partial G_v$ the preimage of a vertex $v$ of $T$ under this projection. This notation is consistent with the fact that $\partial G_v$ can be identified with the Gromov boundary of the stabilizing subgroup $G_v$ of $G$ in its action on $T$. \\

\noindent \textbf{Boundary of the Bass-Serre tree.} Denote by $\partial T$ the set of ends of the tree $T$, i.e. the set of infinite geodesic rays in $T$ divided by the equivalence relation obtained by identifying rays when they coincide except at some bounded intial parts. Clearly, $\partial T$ comes with the action of $G$ induced from the action on $T$. \\ 

We define the \textit{boundary} of $\Gamma$, $\delta\Gamma:=\delta_{Stab}\Gamma \,\sqcup\, \partial T$, and the \textit{compactification} of $\Gamma$, $\overline{\Gamma}= \Gamma \,\sqcup\, \delta\Gamma$. This set comes with the action of $G$ (described separately on the parts $\Gamma$ and $\delta\Gamma$), and with the natural map $p: \overline{\Gamma} \ra T \cup \partial T$. The preimage of a vertex $v$ of $T$ is $\Gamma G_v \cup  \partial G_v$, which we identify (at this moment only set theoretically) with the compactification $\overline{\Gamma G_v}$ of the Cayley graph of the corresponding stabilizing subgroup by means of its Gromov boundary.\\

\noindent \textbf{Topology of the compactification $\overline{\Gamma}$.} For a point $x\in \Gamma\subset\overline{\Gamma}$, we set a basis of open neighbourhoods of $x$ in $\Gamma$ to be also a basis of open neighbourhoods of $x$ in $\overline{\Gamma}$. We now define a basis of open neighbourhoods for points of $\delta\Gamma \subset\overline{\Gamma}$. Fix a vertex $v_0$ in the Bass-Serre tree $T$.

\begin{itemize}

\item Let $\xi \in \delta_{Stab}\Gamma$ and let $v$ be the vertex of $T$ such that $\xi \in \partial G_v$. Let $U$ be a neighbourhood of $\xi$ in $\overline{\Gamma G_v}$. Define $\widetilde{V}_U$ to be the set of all elements $z \in \overline{\Gamma}$ with projection $p(z)\ne v$, and such that the geodesic in $T$ from $v$ to $p(z)$ starts with an edge $e$ that lifts through $p$ to an edge of ${\Gamma}$ which is glued to $\Gamma G_v$ at a point of $U$ 
. We then set 
$$
V_U(\xi) = U \cup \widetilde{V}_U.
$$


As a basis of neighbourhoods of $\xi$ in  $\overline{\Gamma}$ we take a collection of sets
$V_U(\xi)$ as above, where $U$ runs through some basis of open neighbourhoods of $\xi$ in $\overline{\Gamma G_v}$.

\item Let $\eta \in \partial T$ and let $n \geq 1$ be an integer. Let $T_n(\eta)$ be the subtree of $T$ that consists of those elements $x \in T$ for which the geodesic from $v_0$ to $x$ has the same first $n$ edges as the geodesic ray $[v_0, \eta)$. Put also $u_n(\eta)$ to be the vertex at distance $n$ from $v_0$ on the same geodesic ray  $[v_0, \eta)$. We then set 
$$
V_n(\eta) = p^{-1}\big(  T_n(\eta)\setminus\{ u_n(\eta) \} \big).
$$
As a basis of open neighbourhoods of $\eta$ in  $\overline{\Gamma}$ we take the collection of sets
$V_n(\eta)$ for all integer $n \ge1$.


\end{itemize}

We skip a straightforward verification that the above collections of sets satisfy the axioms for basis of open neighbourhoods.

\section{Identification of $\delta\Gamma$ with the Gromov boundary $\partial G$.}

This section is devoted to the proof of Proposition \ref{boundarytopology} below.
This proposition is a reformulation or a slight modification of already
known results, but it cannot be easily justified by simply referring to the literature, for the reasons explained in Remark 3.2.
Thus, for completeness and for the reader's convenience, we provide here a direct proof.

\begin{prop}The space $\delta\Gamma$, equipped with the topology induced from the above described topology on $\overline{\Gamma}$, is a compact metric space which is  $G$-equivariantly homeomorphic to the Gromov boundary $\partial G$ of $G$. 
\label{boundarytopology}
\end{prop}

\begin{remark}
{\rm (1) Proposition \ref{boundarytopology} can be proved by repeating the arguments given by Francois Dahmani in Sections 2 and 3 of \cite{DahmaniCombination}. More precisely, using his arguments one can show that the natural action of $G$ on $\delta\Gamma$ is a uniform convergence action. (This does not follow directly from the results of Dahmani, as he considers graphs of groups with infinite edge groups only.) By a theorem of Bowditch \cite{BowditchTopologicalCharacterization}, it follows that $\delta\Gamma$ is then $G$-equivariantly homeomorphic to the Gromov boundary of $G$. A convincing account of the fact that Dahmani's arguments yield Proposition \ref{boundarytopology} would be rather long, even though no essential adaptations are required.

(2) Proposition \ref{boundarytopology} (or at least its main assertion not dealing with $G$-equivariance) is a reformulation (and a restriction to hyperbolic groups) of a result by Wolfgang Woess (see \cite{WoessMartinBoundary} or Section 26.B in \cite{WoessRandomWalks}). Woess studies an object called the {\it Martin boundary} of a graph (or group), which is known to coincide with the Gromov boundary if the graph (or the group) is hyperbolic, see e.g. Section 27 of \cite{WoessRandomWalks}. He describes the Martin boundary of the free product of two groups, in terms of Martin boundaries of the factors. His description uses a slightly different language than ours, and the topology is introduced in terms of convergent sequences rather than neighbourhoods of points. To justify precisely our statement of Proposition \ref{boundarytopology} by referring to the result of Woess, one needs to provide a translation of his setting to ours, and this cannot be done in one sentence. 

(3) Proposition \ref{boundarytopology} follows also from a much more general result contained in the first author's PhD thesis \cite{MartinPhD} (not yet defended or published), namely from Corollary 5.2.8 in his thesis.} 

\end{remark}

We now pass to the proof of Proposition \ref{boundarytopology}. 

Recall that an edge $e$ of $T$ canonically lifts to an edge $\widetilde{e}$ of $\Gamma$ such that the projection $\Gamma \ra T$ maps $\widetilde e$ onto $e$.  We call $\widetilde e$ the \textit{lift} of $e$. For an edge $e= [v,v']$ of $T$, we call the points $\widetilde{e} \cap \Gamma G_v$ and $\widetilde{e} \cap \Gamma G_{v'}$ the \textit{attaching points} of $\widetilde{e}$.

Observe that
we can obtain a Cayley graph for $G$ by collapsing to points lifts of all edges of $T$ in $\Gamma$. Moreover, the corresponding quotient map is $G$-equivariant.  In particular, $\Gamma$ is $G$-equivariantly quasi-isometric to $G$, so $\Gamma$ is a hyperbolic space and the Gromov boundaries $\partial\Gamma$ and $\partial G$ are $G$-equivariantly homeomorphic.
It is thus sufficient to construct a $G$-equivariant homeomorphism $\partial\Gamma\to\delta\Gamma$.

\medskip
\noindent
{\bf Gromov boundary $\partial\Gamma$.}
We now describe the Gromov boundary $\partial\Gamma$. As a set, it consists of equivalence classes of geodesic rays in $\Gamma$ (started at arbitrary vertices), where two geodesic rays are equivalent if they remain at finite distance from one another. Such geodesic rays $\rho$ 
are easily seen to have one of the following two forms:

\begin{itemize}
\item[(r$_1$)] 
$\rho$ is the concatenation of an infinite sequence of polygonal paths of the form
$$\rho=[u_1,w_1]\tau_1[u_2,w_2]\tau_2\ldots
\leqno{(3.1.1)}
$$
where each $[u_i,w_i]$ is a geodesic path in a single subgraph $\Gamma G_{v_i}$ in $\Gamma$ ($[u_1,w_1]$ may be reduced to a trivial path), and where each $\tau_i$ is the lift of an edge $e_i$ of $T$; in particular, the concatenation $e_1e_2\ldots$ defines a geodesic ray in $T$;

\item[(r$_2$)] 
$\rho$ is the concatenation of a finite sequence of polygonal paths of the form
$$
\rho=[u_1,w_1]\tau_1\ldots[u_{k-1},w_{k-1}]\tau_{k-1}\rho_{k}
\leqno{(3.1.2)}
$$
for  $k\geq0$, where the paths $[u_i,w_i]$ and the edges $\tau_i$ are as in (r$_1$), and where $\rho_k$ is a geodesic ray in $\Gamma G_{v_k}$ based at the attaching point of $\tau_{k-1}$.
\end{itemize}

\noindent
An easy observation is that two geodesic rays $\rho,\rho'$ in $\Gamma G$
are equivalent if and only if:

\begin{itemize}
\item[(e$_1$)] either they both have form (3.1.1) and their corresponding infinite sequence of edges $(\tau_i)_{i\ge1}$ and $(\tau_i')_{i\ge1}$ eventually coincide (i.e. they coincide after deleting some finite intial parts), or

\item[(e$_2$)] they both have form (3.1.2),
and in this case their terminal geodesic rays $\rho_{k},\rho'_{k'}$ belong to the same subgraph $\Gamma G_{v_k}=\Gamma G_{v'_{k'}}$ and are equivalent in it.
\end{itemize}

\medskip
\noindent
{\bf The map $h:\partial\Gamma\to\delta\Gamma$.}
It follows from (e$_1$) and (e$_2$) that $\partial\Gamma$ is in bijective correspondence with the set $\delta_{Stab}\Gamma \cup \partial T$, that is, with $\delta\Gamma$. More precisely, if $\xi\in\partial\Gamma$ is an equivalence class as in $(e_1)$, it corresponds to an end of $T$ induced by the corresponding class of rays $e_1e_2\dots$ in $T$. If $\xi$  is an equivalence class as in $(e_2)$, it corresponds to the point of the Gromov boundary $\partial G_{v_k}$ represented by the equivalence class of geodesic rays $\rho_k$. We denote by $h$ the associated natural bijection $\partial\Gamma \ra \delta\Gamma$. This map is easily seen to be $G$-equivariant. Since $\partial\Gamma=\partial G$ is compact, to get Proposition \ref{boundarytopology} it is enough to prove that $\delta\Gamma$ is Hausdorff and $h$ is continuous.

\begin{lem}
The space $\delta\Gamma$ is Hausdorff.
\label{Hausdorff}
\end{lem}

\begin{proof}
Since the points in $\delta\Gamma$ are of two different natures, there are three cases to consider.\\ 

\textit{Case 1:} Let $\eta, \eta'$ be two distinct points of $\partial T$. Let $n$ be the length of the maximal common subsegment in the geodesic rays $[v_0,\eta)$ and $[v_0,\eta')$. It follows from the definition of the topology of $\delta\Gamma$ that the associated neighbourhoods $V_{n+1}(\eta)$ and $V_{n+1}(\eta')$ are disjoint. \\

\textit{Case 2:} Let $\eta \in \partial T$ and $\xi \in \delta_{Stab}\Gamma$. Let $v$ be the vertex of $T$ such that $\xi \in \partial G_v$. The intersection of the geodesic rays $[v_0, \eta) \cap [v, \eta)$ is a geodesic ray of the form $[w, \eta)$ for some vertex $w$ of $T$. Let $n$ be the combinatorial distance between $v_0$ and $w$ in $T$. Let $e$ be the first edge on the geodesic ray $[v, \eta)$ and $\widetilde{e}$ its lift to $\Gamma$. Let $U$ be a neighbourhood of $\xi$ in $\overline{\Gamma G_v}$ that misses the attaching point of $\widetilde{e}$. The definition of the topology of $\delta\Gamma$ then implies that the neighbourhoods $V_U(\xi)$ and $V_{n+1}(\eta)$ are disjoint. \\

\textit{Case 3:} Let $\xi, \xi'$ be two distinct points of $\delta_{Stab}\Gamma$. Let $v,v'$ be the vertices of $T$ such that $\xi \in \partial G_v$ and $\xi' \in \partial G_{v'}$.

If $v=v'$, then as $\overline{\Gamma G_v}$ is Hausdorff, we can choose disjoint neighbourhoods $U, U'$ of $\xi, \xi'$ in $\overline{EG_v}$, which yields disjoint neighbourhoods $V_U(\xi)$ and $V_{U'}(\xi')$ in $\delta\Gamma$. 

If $v \neq v'$, let $e$ (resp. $e'$) be the edge of the geodesic segment $[v,v']$ which contains $v$ (resp. $v'$). Let $\widetilde{e}, \widetilde{e}'$ be their lifts to $\Gamma$. We choose a neighbourhood  $U$ (resp. $U'$) of $\xi$ (resp. $\xi'$) in $\overline{\Gamma G_v}$ (resp. $\overline{\Gamma G_{v'}}$) which misses the attaching point of $\widetilde{e}$ (resp. $\widetilde{e}'$). The definition of the topology of $\delta\Gamma$ then implies that the neighbourhoods $V_U(\xi)$ and $V_{U'}(\xi')$ are disjoint.
\end{proof}

\noindent
{\bf The map $h$ is continuous.}
As $\partial\Gamma$ is metrisable, it is enough to prove that $h$ is sequentially continuous. To do this, we fix a vertex $u_0\in \Gamma G_{v_0}$ (where $v_0$ is our chosen base vertex of $T$), and we view it as a basepoint of $\Gamma$. Points $z\in \partial\Gamma$ are then represented by (the equivalence classes of) geodesic rays started at $u_0$ and convergence of sequences in $\partial\Gamma$ is characterized by
$$
z_n\to z \hbox{\quad iff \quad} d(u_0,[z,z_n])\to\infty,
\leqno{(3.1.3)}
$$  
where $d$ is the standard geodesic metric on $\Gamma$ and $[z,z_n]$ is any geodesic in $\Gamma$ connecting the corresponding pair of boundary points (see e.g. Remark 3.17(6) and Exercise 3.18(3) on p. 433 in \cite{BridsonHaefliger}). 

We start by a useful lemma, in the statement of which we refer to the natural map $\partial\Gamma \ra {\cal V}(T) \cup \partial T$ resulting from conditions ($e_1$) and ($e_2$), which we again call a \textit{projection}.

\begin{lem}
Let $z_n: n\in \mathbb{N}$ and $z$ be points of $\delta\Gamma$, $x_n: n\in \mathbb{N}$ and  $x$ their projections in ${\cal V}(T) \cup \partial T$. Suppose that the intersection $[v_0, x_n] \cap [v_0, x]$ is a fixed (i.e. independent of $n$) geodesic segment in $T$ that is strictly contained in both $[v_0, x_n] $ and $ [v_0, x]$. Then the sequence of distances $(d(u_0,[z,z_n]))$ is bounded.
\label{LemmaGromovProduct}
\end{lem}

\begin{proof}  The intersection of geodesics $[v_0, x_n]$ and $[v_0,x]$ can be written as the concatenation $e_1\ldots e_m$ of a finite sequence of edges of $T$. Then, in accordance with (3.1.1) or (3.1.2), a geodesic ray in the equivalence class $z$ started at $u_0$ has a form  
$$
[u_0,w_1]\widetilde{e_1}[u_2,w_2] \ldots \widetilde{e_m}[u_{m+1},w_{m+1}]\tau_{m+1}\ldots
$$ 
where $\widetilde{e_i}$ is the lift of $e_i$, and where the appearance of $\tau_{m+1}$ follows from the assumption that $[v_0, x_n] \cap [v_0, x]$ is strictly contained in both intersected geodesics. Since $\Gamma$ has a structure of tree of spaces, for any $n$  every geodesic $[z,z_n]$ passes through the vertex $w_{m+1}$. Thus we have $d(u_0,[z,z_n])\le d(u_0,w_{m+1})$, hence the lemma.
\end{proof}

\begin{cor} Let $\eta\in\partial\Gamma$ be a point such that $h(\eta) \in \partial T$ (i.e. geodesic rays representing $\eta$ have form (2.1.1)), and let $(z_n)$ be a sequence converging to $\eta$ in $\partial\Gamma$. Then $h(z_n)$ converges to $h(\eta)$ in $\delta\Gamma$. \qed
\label{continuity1}
\end{cor}

\begin{proof}
Let $x_n \in {\cal V}(T) \cup \partial T$ be the projections of $z_n$. Since by (3.1.3) we have $d(u_0,[\eta,z_n]) \ra \infty$, Lemma \ref{LemmaGromovProduct} implies that $x_n$ converges to $\eta$ in $T \cup \partial T$. By definition of the topology of $\overline{\Gamma}$, this implies that $h(z_n)$ converges to $h(\eta)$ in $\delta\Gamma$.
\end{proof}

\begin{lem}
Let $\xi\in\partial\Gamma$ be a point such that $h(\xi)\in \delta_{Stab}\Gamma$ (i.e. geodesic rays representing $\xi$ have form (3.1.2)), and let $(z_n)$ be a sequence of elements of $\partial\Gamma$ converging to $\xi$. Then $h(z_n)$ converges to  $h(\xi)$ in $\delta\Gamma$.
\label{continuity2}
\end{lem}

\begin{proof}
Let $v$ be the vertex of $T$ such that $h(\xi) \in \partial G_v$ (i.e. $v$ is the image of $\xi$ through the projection $\partial\Gamma\to {\cal V}(T)\cup\partial T$). Denote by
$$
[u_0,w_1]\tau_1\ldots[u_{k},w_{k}]\tau_{k}\rho_{k+1}
$$
a geodesic ray in the equivalence class $\xi$, where $\rho_{k+1}$ is a geodesic ray in $\Gamma G_v$ based at the attaching point of $\tau_{k}$ in $\Gamma G_{v}$, which we denote $u_{k+1}$. Let $x_n \in {\cal V}(T) \cup \partial T$ be the projections of $z_n$. By Lemma \ref{LemmaGromovProduct}, we have that for all sufficiently large $n$ the geodesic $[v_0, x_n]$ contains $[v_0,v]$.  For such large enough $n$, put $y_n:=z_n$ if $x_n=v$. If $x_n\ne v$, the equivalence class $z_n$ consists of geodesic rays of form
$$
[u_0,w_1]\tau_1\ldots[u_{k},w_{k}]\tau_{k}[u_{k+1},w^n_{k+1}]\tau_{k+1}\dots,
$$ 
where the vertex $w^n_{k+1}$ is uniquely determined by $z_n$ (i.e. it is common for all geodesic rays as above). For such $n$ put $y_n:=w^n_{k+1}$. It follows from the structure of $\Gamma$ that for all $n$ as above we have
$$
d(u_0,[\xi,z_n])=d(u_0,u_{k+1})+d(u_{k+1},[\xi,y_n]),
$$
and thus, by (3.1.3), $d(u_{k+1},[\xi,y_n])\to\infty$. Since $\Gamma G_v$ is geodesically convex in $\Gamma$, the latter convergence implies that $y_n\to h(\xi)$ in $\overline{\Gamma G_v}$ (here we identify those $y_n$ for which $x_n=v$ with elements $h(y_n)\in\partial(\Gamma G_v)=\partial G_v$). From description of neighbourhoods of $h(\xi)$ in the topology of $\overline\Gamma$ it follows that almost all of $h(z_n)$ belong to any such neighborhood, which clearly means that $h(z_n)\to h(\xi)$.
\end{proof}

It follows from Corollary \ref{continuity1} and Lemma \ref{continuity2} that $h$ is continuous, which completes the proof of Proposition \ref{boundarytopology}.


\section{Homeomorphism type of the boundary of a free product.}

This section is devoted to the proof of the following special case of Theorem \ref{Thm1}.

\begin{thm}
For $i=1,2$, let $H_i=A_i*B_i$ be the free product of infinite hyperbolic groups, and suppose that we have homeomorphisms $\partial A_1\cong \partial A_2$ and $\partial B_1\cong \partial B_2$ between Gromov boundaries. Then the boundaries $\partial H_i$ are homeomorphic.
\label{homeofreeprod}
\end{thm}

Before starting the proof, which occupies the next two subsections, we fix some notation.
Let $h_A: \partial A_1 \xrightarrow{\sim} \partial A_2$ and $h_B: \partial B_1 \xrightarrow{\sim} \partial B_2$ be some fixed homeomorphisms of boundaries resulting from the assumptions of the theorem. We denote by $h: \partial A_1 \cup \partial B_1 \ra \partial A_2 \cup \partial B_2$ the homeomorphism induced by $h_A$ and $h_B$.

In what follows, we use the same notation as in the previous two sections. We denote by $T_1$ (respectively $T_2$) the Bass-Serre tree associated to the product $A_1 * B_1$ (respectively $A_2 * B_2$). We also denote by $\Gamma_1=\Gamma(A_1,B_1)$ and $\Gamma_2=\Gamma(A_2,B_2)$ the corresponding trees of Cayley graphs, as described in Subsection 2.1.  

\subsection{Specifying an isomorphism between Bass-Serre trees.}

Since we deal with infinite finitely generated factor groups, the Bass-Serre tree of each of the splittings $H_i=A_i*B_i$ is abstractly isomorphic to the unique (up to simplicial isomorphism) simplicial tree with infinite countable valence at every vertex. Thus the associated Bass-Serre trees of the two splittings are abstractly isomorphic. In this subsection, we specify an isomorphism between those trees with some additional properties. We will use this isomorphism to describe a homeomorphism between the two boundaries $\partial H_i$. 

Recall that given a hyperbolic group $G$, the natural associated topology turns $G \cup  \partial G$  (where $\partial G$ is the Gromov boundary) into a compact metric space.

\begin{lem} Let $G, H$ be two infinite hyperbolic groups, and let
$f:\partial G\to\partial H$ be a homeo\-morphism
between their Gromov boundaries.
Then there is a bijection $b:G\to H$ such that $b(1)=1$ and
$b\cup f:G\cup\partial G \to H\cup\partial H$
is a homeomorphism.
\label{extension}
\end{lem}

\begin{proof}
Let $d_G, d_H$ be metrics on $G \cup \partial G, H \cup \partial H$. Choose any map $\pi:G\ra\partial G$ such that
$\lim_g d_G(g,\pi(g))=0$.
Note that it is sufficient to choose $b$ so that
$$(*) ~~\lim_g d_H\big(b(g),f(\pi(g))\big)=0.$$
Indeed, if we consider a sequence $(g_k)$ of $G$ that converges to $\xi\in\partial G$,
then $\pi(g_k)$ converges to $\xi$ by definition of $\pi$,
 hence $f(\pi(g_k))$ converges to $f(\xi)$ by continuity of
$f$.
Since, by $(*)$, $\lim d_H\big(f(\pi(g_k)), b(g_k)\big)=0$,
it follows that $b(g_k)$ converges to $f(\xi)$,
which shows that $b\cup f:G\cup\partial G \to H\cup\partial H$
is a homeomorphism.

To choose $b$ as above, put $b(1)=1$, order $G\setminus\{1\}$ and $H\setminus\{1\}$ into sequences $(g_k), (h_k)$,
and iterate the following two steps alternately:

\textbf{Step 1.}
Consider the smallest $k$ for which $b(g_k)$ has not yet been defined.
Choose some integer $l$ such that $d_H\big(h_l,f(\pi(g_k))\big) < 1/k$
and such that $h_l$ was not yet chosen as image of any $g_i$ (such an element exists because $H$ is dense in $H\cup\partial H$). Set $b(g_k)=h_l$.

\textbf{Step 2.}
Consider the smallest $k$ for which $h_k$ was not yet chosen as the image
of any $g$, and choose any $g\in G\setminus\{1\}$ such that 
$d_H\big(h_k,f(\pi(g))\big) < d_H\big(h_k,\partial H\big) + 1/k$
and such that $b$ has not yet been defined on $g$ (such an element $g$ exists since $\pi(G)$ is dense in $\partial G$, hence $f(\pi(G))$ is dense in $\partial H$).
Set $b(g)=h_k$.

Then $b$ is obviously a bijection. A straightforward verification shows that $b$ satisfies property ($*$).
\end{proof}

As a consequence, viewing each group as the vertex set of its Cayley graph, and the group unit as the base vertex of this graph, we get the following.

\begin{cor}
There exists a bijection $\alpha: A_1 \ra A_2$ (respectively $\beta: B_1 \ra B_2$) such that $\alpha(1)=1$ (respectively $\beta(1)=1$), and the following property holds: 

Let $\xi \in \partial A_2$ (respectively $\xi\in\partial B_2$) and let $U_2$ be an open neighbourhood of $\xi$ in $\overline{\Gamma A_2}$ (respectively, in $\overline{\Gamma B_2}$). Then there exists a neighbourhood $U_1$ of $h^{-1}(\xi)$ in $\overline{\Gamma A_1}$ (respectively $\overline{\Gamma B_1}$) such that for every element $a \in A_1\cap U_1$ (respectively $b \in B_1\cap U_1$)  we have $\alpha(a) \in U_2$ (respectively $\beta(b)\in U_2$). 
\label{technicallemma}
\end{cor}

\begin{proof}
We prove only the part concernig existence of $\alpha$ (the corresponding part
for $\beta$ clearly follows by the same argument). Let $\alpha:A_1\to A_2$
be any bijection as in Lemma \ref{extension}, i.e. such that $\bar\alpha:=\alpha\cup h_A:A_1\cup\partial A_1\to A_2\cup\partial A_2$ is a homeomorphism. We will show that this $\alpha$ is as required. Viewing $A_i\cup\partial A_i$ as subspaces in $\Gamma A_i$, we get  that the preimage $\bar\alpha^{-1}(U_2)$ is an open neighbourhood of $h^{-1}(\xi)$ in $A_1\cup\partial A_1$. Since the complement of this set, $K=(A_1\cup\partial A_1)\setminus \bar\alpha^{-1}(U_2)$, is compact, and since $h^{-1}(\xi)\in \overline{\Gamma A_1}\setminus K$, there is a neighbourhood $U_1$ of $h^{-1}(\xi)$ in $\overline{\Gamma A_1}$ disjoint from $K$, and consequently such that
$U_1\subset \bar\alpha^{-1}(U_2)$. One verifies directly that this $U_1$ is as required.
\end{proof}

Given bijections $\alpha,\beta$ as above, we can now define a specific isomorphism $\iota: T_1 \ra T_2$ as follows. For $i=1, 2$, 
let $\tau_i=[u_i,v_i]$ be the edge of $T_i$ such that $u_i,v_i$ are stabilized by $A_i,B_i$ respectively.
Recall that each element $g\in H_1$ can be expressed uniquely, in {\it reduced form}, as
$g=a_1b_1\ldots a_nb_n$, with $a_j\in A_1\setminus \{1\}, b_j\in B_1\setminus\{1\}$, allowing also that $a_1=1$ and that $b_n=1$.
It is not hard to realize that there is an isomorphism $T_1\to T_2$
which maps any edge $a_1b_1\ldots a_nb_n\cdot\tau_1$
to the edge $\alpha(a_1)\beta(b_1)\ldots \alpha(a_n)\beta(b_n)\cdot\tau_2$,
and it is obviously unique. 
We denote this isomorphism by $\iota$, and we observe that $\iota(v_1)=v_2$.

\subsection{Construction of the homeomorphism $\partial(A_1*B_1)\to\partial(A_2*B_2)$.}

To show that the Gromov boundaries $\partial H_1, \partial H_2$ are homeomorphic, it is sufficient to describe a homeomorphism $\delta\Gamma_1\to\delta\Gamma_2$. To do this, we need the notion of a {\it reduced representative} of an element $[g,\xi]\in\delta_{Stab}\Gamma_1$. Note that the set of all representatives
of $[g,\xi]$ has the form $(ga,a^{-1}\xi):a\in A_1$ when $\xi\in\partial A_1$,
and the form $(gb,b^{-1}\xi):b\in B_1$ when $\xi\in\partial B_1$. In any case we choose for the reduced representatitive this pair $(g',\xi')$ in which
the reduced form of $g'$ is the simplest one among the elements of the corresponding coset. More precisely, when $\xi\in\partial A_1$,
we choose the unique $ga=a_1b_1\dots a_nb_n$ for which $n$ is the smallest possible (in which case we have $b_n\ne1$), and when $\xi\in\partial B_1$, we choose the unique $gb=a_1b_1\dots a_nb_n$ for which $b_n=1$.

\medskip
Let $\alpha: A_1 \ra A_2$, $\beta:B_1 \ra B_2$, $\tau_1$, $v_1$ and $\iota:T_1\to T_2$
be as in the previous subsection.  
We now define a map $F: \delta\Gamma_1 \ra \delta\Gamma_2$ as follows.  

\begin{itemize}
\item Let $[g, \xi]$ be an element of $\delta_{Stab}\Gamma_1$, such that $(g, \xi)$ is its reduced representative. Write the reduced expression $g=a_1b_1\ldots a_nb_n$. We set
$$ F( [a_1b_1\ldots a_nb_n, \xi]) = [\alpha(a_1)\beta(b_1)\ldots\alpha(a_n)\beta(b_n), h(\xi)].
$$

\item Let $\eta$ be an element of $\partial T$. We can represent it as an infinite word $\eta = a_1b_1 \ldots$,  such that for each $n$ the subword consisting of its first $n$ letters corresponds to the $n$-th edge of the geodesic from $v_1$ to $\eta$ (via the correspondence $g\to g\cdot\tau_1$). We set 
$$
F( a_1b_1 \ldots ) = \alpha(a_1)\beta(b_1) \ldots 
$$ 
where the infinite word on the right is similarly interpreted as a geoedesic ray in $T_2$ started at $v_2$.
Note that this amounts to defining the restriction of $F$ to $\partial T_1$ as the map  $\partial T_1 \ra \partial T_2$ induced by the isomorphism $\iota: T_1 \ra T_2$.
\end{itemize}

Compatibility of the above described map $F$ with the isomorphism $\iota$ manifests also by the following: for any vertex $v$ of the tree $T_1$ the subset $\partial(H_1)_v\subset\delta\Gamma_1$ is mapped by $F$ bijectively to the subset $\partial(H_2)_{\iota(v)}\subset \delta\Gamma_2$.

\begin{prop}
The map $F$ is a homeomorphism. 
\label{homeomorphism}
\end{prop}

\begin{proof} The map $F$ is clearly bijective, so we need to prove that it is continuous. 

Let $\xi \in \delta_{Stab }\Gamma_1$ and let $v$ be the vertex of $T_1$ such that $\xi \in \partial (H_1)_v$. Let $U_2$ be a neighbourhood of $F(\xi)$ in $\overline{\Gamma(H_2)_{\iota(v)}}$. Let $U_1$ be a neighbourhood of $\xi$ in $\overline{\Gamma(H_1)_{v}}$ as in Corollary \ref{technicallemma} (where we identify canonically $\overline{\Gamma(H_1)_{v}}$ with $\overline{\Gamma A_1}$ or $\overline{\Gamma B_1}$,
and $\overline{\Gamma(H_2)_{\iota(v)}}$ with $\overline{\Gamma A_2}$ or $\overline{\Gamma B_2}$). It is straightforward to check that $F\big( V_{U_1}(\xi)\cap\delta\Gamma_1 \big) \subset V_{U_2}(F(\xi))\cap\delta\Gamma_2$, and hence $F$ is continuous at $\xi$.

Let $\eta \in \partial T_1$. The isomorphism $\iota: T_1 \ra T_2$ extends to a map $\iota: T_1 \cup \partial T_1 \ra T_2 \cup \partial T_2$ so that $\iota|_{\partial T_1}$ coincides with $F|_{\partial T_1}$. Let $n \geq 1$ be an integer, and consider the subtree $(T_2)_n\big( F(\eta) \big) \subset T_2$, defined with respect to the base vertex $v_2$. Then $\iota^{-1} \big( (T_2)_n\big( F(\eta) \big) \big)=(T_1)_n(\eta)$, where the latter subtree is defined with respect to the base vertex $v_1$. 
It is then straightforward to check that $F\big( V_n(\eta)\cap\delta\Gamma_1 \big)= V_n\big(F(\eta)\big)\cap\delta\Gamma_2$, and hence $F$ is continuous at $\eta$. 
\end{proof}

This completes the proof of Theorem \ref{homeofreeprod}.

\section{The proof of Theorem \ref{Thm1}.}

The proof is based on Proposition \ref{boundarytopology} and the following result of P. Papasoglu and K. Whyte (see \cite{PapasogluWhyteQI}).

\begin{thm} [Papasoglu-Whyte]
Let $G,H$ be finitely generated groups with infinitely many ends and let ${\cal G},{\cal H}$ be their graph of groups decompositions with all edge groups finite. If $\cal G$, $\cal H$ have the same set of quasi-isometry types of infinite vertex groups (without multiplicities) then $G$ and $H$ are quasi-isometric. \qed
\label{PapasogluWhyte}
\end{thm}

Theorem \ref{PapasogluWhyte} implies in particular that if $G$ is any group distinct from $\bbZ_2$ then the groups $G*\bbZ \cong G*_1$ and $G*G$ are quasi-isometric.
In view of this, Theorem \ref{PapasogluWhyte} has also the following consequence.

\begin{cor}
Under assumptions and notation of Theorem \ref{PapasogluWhyte}, if groups $G_1,\dots,G_m$ represent all quasi-isometry types of those infinite vertex groups of $\cal G$ which are not virtually cyclic then

\begin{itemize}

\item if $m=0$ then $\pi_1{\cal G}$ is quasi-isometric to  the free group $F_2$;

\item if $m=1$ then $\pi_1{\cal G}$ is quasi-isometric to $G_1*G_1$;

\item if $m>1$ then $\pi_1{\cal G}$ is quasi-isometric to $G_1*\dots*G_m$. \qed

\end{itemize}
\label{PapasogluWhyteCor}
\end{cor}

\noindent
{\bf Proof of Theorem \ref{Thm1}.}
Recall that, by the assumption, the groups $\pi_1({\cal G}_i)$ have infinitely many ends. Recall also that quasi-isometric hyperbolic groups have homeomorphic Gromov boundaries, and that a group is quasi-isometric to a virtually cyclic group iff it is virtually cyclic. Thus, if the case $m=0$ in Corollary \ref{PapasogluWhyteCor} applies to ${\cal G}_1$, it also applies to ${\cal G}_2$, and the theorem follows in this case from the corollary. In the remaining cases, not changing the quasi-isometry types, we replace the groups $\pi_1({\cal G}_i)$ with the corresponding free products of infinite vertex groups, accordingly with Corollary \ref{PapasogluWhyteCor}. By applying Theorem \ref{PapasogluWhyte} again, we get that for some integer $p\ge2$ and for $i=1,2$  the groups $\pi_1({\cal G}_i)$ are quasi-isometric to some free products of infinite groups $H_i=G_{i,1}*\dots*G_{i,p}$ such that for $j=1,\dots,p$ the groups $G_{1,j},G_{2,j}$ have homeomorphic Gromov boundaries. The theorem then follows by induction from Theorem \ref{homeofreeprod}.\qed

\section{Hyperbolic groups with homeomorphic boundaries.}

In this section, after a few preparations, we prove Theorem \ref{Thm2}.


We start with a brief comment that all of the content of Section 2 can be extended to the case of the free product of finitely many hyperbolic groups.
Let $G=A_1*\dots A_k$ be the free product of hyperbolic groups, and let $\cal G$ be  a graph of groups representing this product. More precisely, $\cal G$
is a graph of groups whose underlying graph $Q$ is a tree with the vertex set $\{ v_1,\dots,v_k \}$, all the edge groups in $\cal G$ are trivial, and for $j=1,\dots,k$ the vertex group $G_{v_j}$ of $\cal G$ is isomorphic to $A_j$.
Let $T$ be the Bass-Serre tree of $\cal G$. We then have a geometric model $\Gamma=\Gamma({\cal G})$ for $G$, in the form of a tree of spaces over $T$, constructed out of the Cayley graphs $\Gamma A_j$ in the way presented in Subsection 2.1. More precisely, we describe first an analog of the graph $X$ from Subsection 2.1, as the union of the Cayley graphs $\Gamma A_j$ connected together with segments joining the base vertices and corresponding to the edges in the tree $Q$. Then $\Gamma=\Gamma({\cal G})$ is the appropriate quotient of the product $G\times X$. 

We next describe the boundary $\delta\Gamma$, as in Subsection 2.2, consisting of two disjoint parts $\partial T$ and $\delta_{Stab}\Gamma$, and equip it with the natural $G$-action and the natural projection $p:\delta\Gamma\to {\cal V}(T)\cup\partial T$. This gives us subsets $\partial G_v\subset\delta\Gamma$ as preimages of vertices $v$ of $T$ under $p$. Finally, we also consider the topology on the compactification $\overline{\Gamma}=\Gamma\cup\delta\Gamma$, described in the same way as in Subsection 2.2, and we mention that Proposition \ref{boundarytopology} literally extends to this more general framework.

We are now ready to state two basic topological properties of $\delta\Gamma$,
whose straightforward proofs we omit.

\begin{lem}
For any $g\in G$ and any $j\in\{ 1,\dots,k \}$, the map $f=f_{g,j}:\partial A_j\to\delta\Gamma$ given by $f(x)=[g,x]$ (where  $[g,x]$ is the image of $(g,x)$ under the quotient map $G\times(\partial A_1\cup\dots\cup \partial A_k)\to\delta_{Stab}\Gamma\subset\delta\Gamma$) is continuous.\qed
\label{technicallemma2}
\end{lem}

Next lemma requires some preparatory notation. Given an oriented edge $e$ of $T$, consider the two subtrees obtained by deleting from $T$ the interior of $e$, and denote by $T_+$ this subtree which contains the terminal vertex of $e$, and by $T_-$ the remaining subtree. Consequently, denote by $\partial T_+$ and $\partial T_-$ the corresponding sets of ends of the subtrees, which we naturally view as subsets of $\partial T$.

\begin{lem}
For any oriented edge $e$ of $T$ the subsets $W^e_+:=p^{-1}(T_+\cup\partial T_+)$ and $W^e_-:=p^{-1}(T_-\cup\partial T_-)$ form an open-closed partition of $\delta\Gamma$.\qed
\label{openclosed}
\end{lem}

We will call the pairs of subsets of $\delta\Gamma$ as in the above lemma {\it the natural halfspaces in $\delta\Gamma$}.
The above two lemmas have the following useful consequences.

\begin{prop} A point of $\eta\in\partial T\subset\delta\Gamma$ is its own connected component in $\delta\Gamma$.
\label{boundarypoints}
\end{prop}

\begin{proof}
A straigtforward observation shows that any point $\xi\in\delta\Gamma$ distinct from $\eta$ can be separated from $\eta$ by a pair of natural halfspaces. In view of Lemma \ref{openclosed}, this clearly implies the assertion.
\end{proof}

\begin{prop}
Let $v$ be a vertex of $T$ which is projected to $v_j$ by the natural projection $T\to Q$, and suppose that the group $A_j$ is 1-ended. Then the subset $\partial G_v=p^{-1}(v)$ is a connected component of $\delta\Gamma$.
\label{connectedcomponent2}
\end{prop}

\begin{proof}
Note that the subset $\partial G_v$ coincides then with the image of the map $f_{g,j}:\partial A_j\to\delta\Gamma$
from Lemma \ref{technicallemma2}, for some appropriately chosen $g$. Since $A_j$ is 1-ended, $\partial A_j$ is connected,
and since $f_{g,j}$ is continuous, its image $\partial G_v$ is also connected.
On the other hand, any point $\xi\in\delta\Gamma\setminus \partial G_v$ can be easily separated from $\partial G_v$ by a pair of natural halfspaces (which are open and closed due to Lemma \ref{openclosed}), which completes the proof. 
\end{proof}

\noindent
{\bf Proof of Theorem \ref{Thm2}.}
By Theorem \ref{PapasogluWhyte}, without changing the quasi-isometry type of $\pi_1({\cal G}_i)$ and the 1-ended vertex groups of ${\cal G}_i$, we can assume that ${\cal G}_i$ represent free products $G_i=A_{i,1}*\dots* A_{i,k_i}$ of 1-ended groups $A_{i,j}$. By Propositions \ref{boundarypoints} and \ref{connectedcomponent2} (accompanied with the above mentioned extension of Proposition \ref{boundarytopology}), the connected components of $\partial G_i$ are single points and subspaces homeomorphic to boundaries $\partial A_{i,j}$ for all $j\in\{ 1,\dots,k_i \}$. Since a homeomorphism $\partial G_1= \partial\pi_1({\cal G}_1)\to\partial\pi_1({\cal G}_2)=\partial G_2$ clearly maps components homeomorphically onto components, we get $h({\cal G}_1)=h({\cal G}_2)$, which concludes the proof.\qed

\bibliographystyle{alpha}
\bibliography{HomeoBoundaries}

\noindent Alexandre Martin, IRMA, Université de Strasbourg, 7 rue René-Descartes, 67084 Strasbourg Cedex, France.

\noindent E-mail: martin@math.unistra.fr\\

\noindent Jacek Świątkowski, Instytut Matematyczny, Uniwersytet Wroc\l awski, pl. Grunwaldzki 2/4, 50-384 Wroc\l aw, Poland

\noindent E-mail: Jacek.Swiatkowski@math.uni.wroc.pl
\end{document}